\documentclass[12pt]{article}

\usepackage{amssymb}
\usepackage{amsmath}
\usepackage{graphicx}
\usepackage{algorithmic}
\usepackage{etoolbox}
\usepackage{hyperref}
\usepackage{amsthm}
\usepackage{setspace}
	
\theoremstyle{plain}
\newtheorem{theorem}{Theorem}
\newtheorem{corollary}[theorem]{Corollary}

\newtheorem{proposition}[theorem]{Proposition}

\newtheorem{claim}[theorem]{Claim}

\theoremstyle{definition}
\newtheorem{definition}[theorem]{Definition}
\newtheorem{remark}[theorem]{Remark}
\newtheorem{example}[theorem]{Example}
\newtheorem{algorithm}[theorem]{Algorithm}
\numberwithin{equation}{section}

\newenvironment{Macaulay2}{ \begin{spacing}{0.4} 
\smallskip } { \smallskip 
\end{spacing} }

\newcommand{\B}[1]{\mathbb #1}

\newcommand{\F}[1]{\mathfrak #1}

\newcommand{\BF}[1]{\mathbf #1}

\DeclareMathOperator{\lcm}{lcm}
\DeclareMathOperator{\NF}{NF}

\newcommand{\ring}{R_0}
\newcommand{\ringh}{\tilde{R}_0}
\newcommand{\dual}{D_0}
\newcommand{\sdual}{S_0}
\newcommand{\dualh}{\tilde{D}_0}
\newcommand{\init}{\operatorname{in}_>}
\newcommand{\Ih}{\tilde{I}}
\newcommand{\LT}{\operatorname{in}_>}
\newcommand{\LTg}{\operatorname{in}_\succ}
\newcommand{\D}{{\bf d}}
\newcommand{\ideal}[1]{\langle #1 \rangle}
\newcommand{\deh}[1]{\psi( #1 )}
\newcommand{\dehw}[1]{\psi( #1 )}

\author{Robert Krone}

\title{Numerical Algorithms for Dual Bases of Positive-Dimensional Ideals}

\begin{document}
\maketitle

\begin{abstract}
An ideal of a local polynomial ring can be described by calculating a standard basis with respect to a local monomial ordering.  However the usual standard basis algorithms are not numerically stable.  A numerically stable approach to describing the ideal is by finding the space of dual functionals that annihilate it, which reduces the problem to one of linear algebra.  There are several known algorithms for finding the truncated dual up to any specified degree, which is useful for describing zero-dimensional ideals.  We present a stopping criterion for positive-dimensional cases based on homogenization that guarantees all generators of the initial monomial ideal are found.  This has applications for calculating Hilbert functions.
\end{abstract}



A Gr\"obner basis for a polynomial ideal provides a wealth of computational information, for example the dimension of the ideal, its Hilbert function, a way to answer the ideal membership question, and more.  Computing a Gr\"obner basis is a well understood problem at least in the setting of exact computation, for example using the Buchberger algorithm.  Roughly the same can be said about ideals in a local polynomial ring.  In the exact setting we can compute a {\em standard basis} (the local equivalent of a Gr\"obner basis) using variations of Buchberger's algorithm such as those using Mora's Normal Form algorithm.  A treatment of ideals in local rings and standard basis algorithms can be found in \cite{Cox-Little-OShea:using} and \cite{Singular-book-02}.

However in many practical situations, using only exact computations becomes infeasible and we are forced to rely on approximate numerical data.  For example many large systems of polynomials can only be solved in practice with numerical algorithms such as homotopy continuation.  We may want to investigate the properties the ideal in the local ring at some solution point, but this point is only known to us approximately.  Although we can approximate the point with arbitrarily high precision, the error can never be entirely eliminated.  In this context the usual algorithms for computing a standard basis are unsuitable because they are not numerically stable.  Even arbitrarily small errors in the initial data can produce results that are combinatorially incorrect, for instance incorrect values of the Hilbert function.  Some of the alternative approaches for computing Hilbert functions, such as Janet basis algorithms \cite{DBLP:journals/jsc/Apel98}, must also be ruled out because they lack numerical stability.  
Other approaches useful in the exact setting can be found in \cite{mora2005solving}.

In a numerical setting, to compute the information provided by a Gr\"obner basis we need to tread carefully because many tools are no longer available to us.  One avenue developed by Bates, et al. \cite{Bates-et-al:local-dimension-test} is to find witness points of the various components of the variety in order to compute dimension of the ideal and other information.  Another approach that can be used in the case of zero-dimensional ideals is computing a border basis of the ideal, which has better numerical stability than Gr\"obner basis computations \cite{KreuzerRobbianoBook2}.  In this paper we will focus on an approach which uses purely local information: computing the local {\em dual space} of the ideal.  The dual space is the vector space of all functionals that annihilate every element of the ideal.  This idea was first developed in the seminal work of Macaulay \cite{Macaulay:modular-systems}.  The dual space of an ideal can provide much of the same information as a standard basis, such as the Hilbert 
function of the ideal, and a test for ideal membership.

There are several algorithms for computing the dual space of an ideal in a local ring, truncated at some degree.  One that will be discussed in this paper is the Dayton-Zeng algorithm presented in \cite{DZ-05}, and another is the Mourrain algorithm presented in \cite{Mourrain:inverse-systems}, although both are based on the ideas of Macaulay.  The numerical advantage to dual space algorithms is that they reduce the problem to finding the kernel of a matrix.  This can be done in a numerically stable way using singular value decomposition (SVD).

These truncated dual space algorithms provide a way to fully characterize the local properties of an ideal when the ideal is zero-dimensional, i.e. the point of interest is an isolated solution.  In this case the dual space has finite dimension, so truncating at a high enough dimension we will find a basis for the whole space.  However, when the ideal is not zero-dimensional or when the dimension is not known a priori, this strategy will fall short.

Our contribution is a method of finding the truncated dual space up to sufficient degree to ensure that the important features of the local ideal are found.  In particular, this means finding an explicit formula for the Hilbert function of the ideal at all values, not just the values up to some finite degree.  The method presented can also be used to recover a standard basis for the ideal, which as far as we know is not possible using existing truncated dual space algorithms alone in a numerical setting.  Additionally these tools can be used to answer the ideal membership test for polynomials up to some bounded degree.  In this way we can describe the local properties of an ideal numerically, using purely local information, for an ideal of any dimension.  We have implemented this method in the Macaulay2 computer algebra system \cite{M2www}.  Our code can be found at \url{http://people.math.gatech.edu/~rkrone3/NHcode.html}.

A potential application for this result is for developing numerical algorithms for computing the primary decomposition of an ideal.  Current algorithms for primary decomposition use elimination theory which relies on Gr\"obner bases.  On the numerical side, there are algorithms for decomposing a variety into irreducible components.  As discussed in \cite{DBLP:conf/issac/Leykin08}, this is done by intersecting with random affine spaces of the correct dimension to collect witness points on various components.  Then homotopy methods are used to decide which witness points belong to the same components.  This is part way to a primary decomposition, since the irreducible components correspond to the minimal associated primes.  However the remaining obstacle is numerically detecting embedded components of the ideal.  Given a point of interest on the variety that has been found numerically, knowledge of the Hilbert function may help decide whether or not the point sits in an embedded component.

In Section \ref{prelim} we describe preliminary information, defining a standard basis and the Hilbert function of an ideal in a local ring, as well as algorithms to calculate them.  In Section \ref{duals} we define the dual space of an ideal, and show how the dual space can be used to recover information about the ideal.  We also describe the Dayton-Zeng and Mourrain algorithms here.  In Section \ref{homog} we show how the homogenization of an ideal $I$ motivates an algorithm for finding the truncated dual up to sufficient degree and Sections \ref{dehomog} and \ref{sylvester} contain our main result: an algorithm for finding the Hilbert function and standard basis for an ideal in a numerically stable way using the dual space.

\section{Preliminaries}\label{prelim}
Let $f_1,\ldots,f_s$ be a system of polynomials in ring $R = k[x_1,\ldots,x_n]$ where $k$ is a field, and let $J = \ideal{f_1,\ldots,f_s}$.  In practice we can assume $k = \B C$ because we are interested primarily in numerical applications.  Suppose the point $b \in \B A^n(k)$ is known to be in the zero set of these polynomials, but $b$ has been calculated numerically so it may not lie exactly on the variety.  We would like to characterize the zero set $\BF V(J)$ in a neighborhood of $b$.

The proper context for answering these local questions is in the local ring at $b$.  Let $R_b$ be the localization of $R$ with respect to the maximal ideal $\F m = \langle x_1 - b_1,\ldots,x_n - b_n \rangle$, so
	\[ R_b = \{f/g |\; f,g\in R,\; g(b) \neq 0\}. \]
	Let $I$ be the extension of $J$ in this local ring $I = JR_b$.  Without loss of generality we will take $b = (0,\ldots,0)$.  This makes calculations simpler, and for $b \neq 0$ we can translate elements of $R_b$ to the origin by substituting each $x_i$ with $x_i + b_i$.  Every $f \in \ring$ can be expressed as a power series which converges in some neighborhood of the origin, so
	\[ f = \sum_{\alpha \in \B N^n} c_\alpha x^\alpha \]
	with each $c_\alpha \in k$.  Here $\alpha = (\alpha_1,\ldots,\alpha_n)$ is a multi-index and $x^\alpha$ denotes the monomial $x_1^{\alpha_1}\cdots x_n^{\alpha_n}$.  Let $|\alpha| = \alpha_1 + \cdots + \alpha_n$, i.e. $|\alpha|$ is the degree of $x^\alpha$.

The ring $\ring$ is equipped with a local order $>$.
\begin{definition}
	A local order is a total order on the monomials of a local ring that is compatible with multiplication, and has $1 > x_i$ for all $x_i$ (in contrast to a monomial order where $1 < x_i$).
\end{definition}
Taking the reverse of any monomial order produces a local order and vice versa.  We will take the local order to be anti-graded, meaning that it respects the degree of the monomials, similar to a graded monomial order.  Let $\LT f$ denote the lead term of $f \in \ring$.  Note that even if $f$ is not polynomial, it still has a well defined lead term when considered as a power series.  Let $\LT I$ denote the initial ideal of $I$, $\LT I = \ideal{\LT f|f \in I}$.

In an exact setting questions about $I$ could be answered by finding a standard basis, which is the local equivalent of a Gr\"obner basis.
\begin{definition}
	Given a local order $>$ on $\ring$, a standard basis $G$ of ideal $I$ is a finite set $G = \{g_1,\ldots,g_r\} \subset I$ with $\ideal{\LT g_1,\ldots,\LT g_r} = \LT I$.
\end{definition}
The simplest algorithm for calculating a standard basis is the Buchberger algorithm (as with a Gr\"obner basis).  For any $f,g \in \ring$ define their S-pair as
	\[ S(f,g) = \frac{\lcm(\LT f,\LT g)}{\LT f}f - \frac{\lcm(\LT f,\LT g)}{\LT g}g. \]
Next define the normal form of an element $f\in \ring$ with respect to a finite set of polynomials $G = \{g_1,\ldots,g_r\} \in \ring$ to be a polynomial
	\[ \NF_G(f) = uf - a_1g_1 -\cdots - a_rg_r \]
for some unit $u$ and polynomials $a_i$ such that $\LT f \geq \LT \NF_G(f)$, $\LT f \geq \LT a_ig_i$, and $\LT \NF_G(f)$ is not divisible by any $\LT f_i$.  Such a polynomial always exists and can be calculated explicitly using Mora's tangent cone algorithm \cite{mora2005solving} (this is the local equivalent of the division algorithm).  Buchberger's algorithm proceeds as follows:  Starting with the generators of $I$, $G = \{f_1,\ldots,f_s\}$, calculate $\NF_G(S(f_i,f_j))$ for each pair $i \neq j$.  If any are non-zero, add them to the set $G$ and repeat the process, otherwise $G$ is a standard basis for $I$.

A standard basis $G$ for $I$ provides answers to many of the questions one might have about the local properties of $J$.  For instance, $f \in I$ if and only if $\NF_G(f) = 0$ so a standard basis provides an algorithmic way to answer the ideal membership question.  From a standard basis we can also calculate the Hilbert function of $I$ which determines dimension of the component of $J$ through $b$, and if $b$ is an isolated solution it determines its multiplicity.
\begin{definition}
	The Hilbert function of ideal $I$ with local order $>$ is $H_I:\B N \to \B N$ where $H_I(d)$ counts the number of monomials with degree $d$ that are not in $\LT I$.
\end{definition}

Consider $\LT I$ in the lattice of monomials.  For $\LT G = \{m_1,\ldots,m_r\}$, each monomial $m_i$ cuts out the cone $C_{m_i}$ of all its monomial multiples, and the monomials in $\LT I$ are exactly $\bigcup_i C_{m_i}$.  The resulting picture is called a ``staircase'' (see Figure \ref{fig:stair}) and the Hilbert function counts the monomials outside the staircase at each degree.  Using the inclusion-exclusion principle and noting that $C_{m_i} \cap C_{m_j} = C_{\lcm(m_i,m_j)}$ we get an explicit combinatorial formula for the Hilbert function:
	\[ H_I(d) = \sum_{S \subset \LT G} (-1)^{|S|} \binom{d - \deg \lcm(S) + n - 1}{n-1} \]
where the binomial coefficients $\binom{p}{q}$ are taken to be 0 for $p < q$.  Note that when defined this way, for fixed $q$ the binomial coefficients are polynomial in $p$ for all $p > 0$, and this polynomial has degree $q$.  As a result, it is clear that the Hilbert function is described by a polynomial for sufficiently large degree.  The regularity is bounded by $\deg \lcm(m_1,\ldots,m_r) - n + 1$, which is when all the binomial coefficients in the sum become polynomial.
\begin{definition}
	The g-corners of ideal $I$ are the monomials that minimally generate $\LT I$.
\end{definition}
The set of g-corners is uniquely determined by $I$ and the local order $>$.  The g-corners of $I$ can easily be found from $\LT G$ for any standard basis $G$, and the g-corners completely determine $H_I$.

\begin{figure}[ht]\label{fig:stair}
  \centering
  \includegraphics[width=.9\columnwidth]{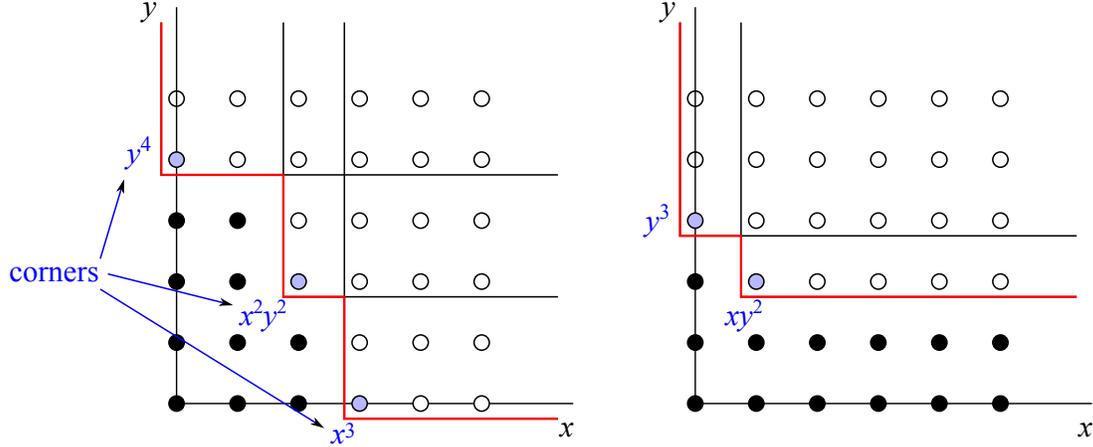}
  \caption{A zero-dimensional staircase, and a one-dimensional staircase.}
\end{figure}

\begin{claim}
	For sufficiently high degree, $H_I$ can be expressed as a polynomial $p$.  If $p = 0$ then $b$ is an isolated point of $J$.  Otherwise the largest dimension component of $J$ through $b$ has dimension $a+1$ where $a$ is the degree of $p$.
\end{claim}

Unfortunately standard basis algorithms such as Buchberger are not numerically stable.  Small error in the calculation of $b$ will cause large errors in the output.  An alternate approach is to use the dual space instead, which can recover the same information as a standard basis and can be calculated with numerical stability.

\section{The Dual Space}\label{duals}
Considering $\ring$ as a vector space, for each monomial $x^\alpha$ there is a linear functional $\partial^\alpha:\ring \to k$ in $\ring^*$ defined by
	\[\partial^\alpha\left(\sum_{\beta \in \B N^n} c_\beta x^\beta \right) = c_\alpha. \]
Let $\dual$ denote the vector space spanned by these monomial dual vectors.  We will refer to $\dual$ as the dual space of $\ring$ even though it is technically a proper subset of $\ring^*$.  By equipping $\dual$ with multiplication $\partial^\alpha \partial^\beta = \partial^{\alpha+\beta}$, it has a $k$-algebra structure $\dual = k[\partial_1,\ldots,\partial_n]$ where $\partial_i$ is the dual element corresponding to $x_i$.  We give this ring a global monomial order $\succ$ which is the reverse of the order on $R_m$, so if $x^\alpha < x^\beta$ then $\partial^\alpha \succ \partial^\beta$.

The dual space is sometimes defined in terms of differentials instead \cite{Mourrain:inverse-systems}.  For $f \in \ring$
	\[ \partial^\alpha(f) = \frac{1}{\prod_i \alpha_i!} \frac{\partial^{\alpha_1}}{\partial x_1^{\alpha_1}} \cdots \frac{\partial^{\alpha_n}}{\partial x_n^{\alpha_n}}f\bigg|_{x=0}. \]

\begin{definition}
	The dual space of the ideal $I$, $\dual[I]$, is the subset of $\dual$ that annihilates $I$.  The truncated dual space of $I$, $\dual^{(d)}[I]$, is the subset of $\dual[I]$ of functionals with lead term of degree $d$ or less.
\end{definition}

\begin{theorem}
\label{compThm}
	Any monomial $x^\alpha$ is in $\LT I$ if and only if the corresponding dual monomial $\partial^\alpha$ is not in $\LTg \dual[I]$.  Equivalently, $\LTg \dual[I] = \dual[\LT I]$.
\end{theorem}
\begin{proof}
	For a proof in the case where $I$ is zero-dimensional, see \cite{LVZ-higher}, Theorem 3.4.  For positive dimensional $I$, fixing any $\alpha$, let $I' = I + \F m^{|\alpha|+1}$ where $\F m$ is the maximal ideal $\ideal{x_1,\ldots,x_n}$.  Note that $I'$ is zero-dimensional because $H_{I'}(c) = 0$ for $c > |\alpha|$, so then $x^\alpha \in \LT I'$ if and only if $\partial^\alpha \notin \LTg \dual[I']$.  Since $x^\alpha$ has lower degree than any element of $\F m^{|\alpha|+1}$, then $x^\alpha \in \LT I'$ if and only if $x^\alpha \in \LT I$.  Additionally $\dual^{(d)}[I] = \dual^{(d)}[I']$ because the elements of $\dual^{(d)}$ and $\F m^{|\alpha|+1}$ have no terms in common.  Therefore $\partial^\alpha \in \LTg \dual[I']$ if and only if $\partial^\alpha \in \LTg \dual[I]$.
\end{proof}

\begin{figure}[ht]
  \centering
  \includegraphics[width=.4\columnwidth]{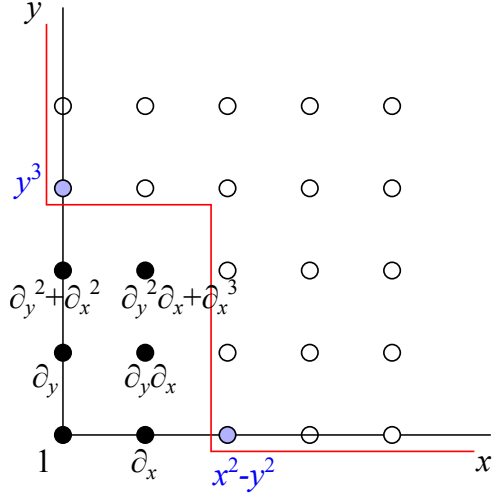}
  \caption{A standard basis (in blue) and dual basis (in black) for $I = \ideal{x^2-y^2,y^3}$.}
\end{figure}

\begin{corollary}
	$H_I(d) = \dim \dual^{(d)}[I] - \dim \dual^{(d-1)}[I]$.
\end{corollary}

\begin{theorem}
	For any $f \in \ring$, $f \in I$ if and only if $p(f) = 0$ for all $p \in \dual[I]$.
\end{theorem}
\begin{proof}
	If $f \in I$ it is clear that $p(f) = 0$ for all $p \in \dual[I]$.  Suppose $f \notin I$, and let $G = \{g_1,\ldots,g_r\}$ be a standard basis of $I$.  Then $f$ can be expressed as
		\[ f = ua_1g_1 + \cdots + ua_rg_r + u\NF_G(f) \]
	where $u$ is a unit, each $a_i$ is a polynomial, and $\LT \NF_G(f) \notin \LT I$.  Let the lead monomial of $\NF_G(f)$ be $x^\alpha$.  Because $u$ is a unit its lead monomial is 1, so the lead monomial of $u\NF_G(f)$ is also $x^\alpha$.  By Theorem \ref{compThm} there is $p \in \dual[I]$ with lead monomial $\partial^\alpha$.  Due to the reverse nature of $>$ and $\succ$, $p$ and $u\NF_G(f)$ have only their lead monomial in common.  Therefore $p(u\NF_G(f)) \neq 0$.  Note that $p(ua_ig_i) = 0$ since $ua_ig_i \in I$ so $p(f) = p(u\NF_G(f)) \neq 0$.
\end{proof}

As a consequence, knowing a basis for the $\dual^{(c)}[I]$ at each degree $c$ up to some finite degree $d$ provides some of the same information as a standard basis.  In particular it reveals the values of the Hilbert function of $I$ for all degrees up to degree $d$.  Algorithms exist for finding a basis for the truncated dual space up to any particular degree.  Two such algorithms are discussed below.  Both reduce the problem to a system of linear constraints.  Finding the kernel of a matrix can be done in a numerically stable way by using singular value decomposition (SVD), which is what makes the dual space approach better suited for numerical situations.  

If $I$ is known a priori to be zero-dimensional then $\dual[I]$ has finite dimension, and it is possible to find an explicit basis for it.  Finding the truncated dual at each successive degree, there will be some $d$ for which $\dual^{(d)}[I] = \dual[I]$, and so $\dim \dual^{(d+1)}[I] = \dim \dual^{(d)}[I]$, at which point we know the entire dual space has been found.  The ideal membership question can then be answered for any element of $\ring$ and the entire Hilbert function can be calculated ($H_I(c) = 0$ for all $c > d$).

If $I$ is not zero-dimensional, or if the dimension is not known, this strategy will not work.  In general the dual space of $I$ is not finite dimensional, so it is not possible to explicitly compute a basis for the entire dual space.  The best we can do is find a truncated dual basis up to any finite degree $d$.  The difficulty with this approach is that it's difficult to tell what degree $d$ one needs to compute to in order to find all the relevant information about the ideal.  Additionally, the truncated dual space can't be used for an ideal membership test.  Given a polynomial $f$, even if $p(f) = 0$ for all $p \in \dual^{(d)}[I]$, it may still be that $f \notin I$.  

Something that the truncated dual space can tell us is what the g-corners of $I$ are up to degree $d$.  Suppose all the g-corners of $I$ are known up to degree $d-1$.  Then the g-corners at degree $d$ are exactly the monomials missing from $\LTg \dual^{(d)}[I]$ that are not multiples of the previously found g-corners.  Computing $\dual^{(d)}[I]$ for successive $d$, if we could determine at what point all the g-corners of $I$ had been found, then we could fully describe the Hilbert function $H_I$, since $H_I$ is determined by the g-corners of $I$.  We will present a way to do so.  These methods will also provide a way to answer the ideal membership test for polynomials up to some fixed degree.

\subsection{Dayton-Zeng Algorithm}
A simple truncated dual space algorithm is one by Dayton and Zeng \cite{DZ-05}, using ideas of Macaulay \cite{Macaulay:modular-systems}.  Given a finite generating set $F$, $I$ considered as a vector space can be expressed as the span of the monomial multiples of the generators:
	\[ I = \operatorname{span}\{x^\alpha f|\; f\in F,\; \alpha \in \B N^n\}. \]
\begin{proposition}
	$\dual^{(d)}[I]$ is the set of functionals $p \in \dual^{(d)}$ satisfying $p(x^\alpha f) = 0$ for all $f \in F$ and $\alpha \in \B N^n$ with $|\alpha|+\deg \LT f \leq d$.
\end{proposition}
Higher degree multiples of the generators need not be considered when calculating $\dual^{(d)}[I]$ because they will not have any terms of degree $d$ or less, so will always be orthogonal to $\dual^{(d)}[I]$.  Let $A_d$ be the set of elements $x^\alpha f$ as above.  To find the subspace of $\dual^{(d)}$ that annihilates all of $A_d$, we construct the Macaulay array $\BF M(F,d)$, which is the coefficient matrix of the elements of $A_d$.  The matrix $\BF M(F,d)$ has entries in $k$ with columns indexed by the monomials of $\ring$ with degree $\leq d$ and a row for each $a \in A_d$.  If $a$ is the $i$th element of $A$ and $x^\beta$ is the $j$th monomial with degree $\leq d$, then $m_{ij} = \partial^\beta(a)$, with $\BF M(F,d) = (m_{ij})$.  The truncated dual corresponds to the kernel of $\BF M(F,d)$.

\begin{example}
\label{MacExample}
	Let $F = \{x-y^3, x^2\} \subset k[x,y]$.  Then the Macaulay array $\BF M(F,3)$ is
		\[ \bordermatrix{ & 1 & \partial_x & \partial_y & \partial_x^2 & \partial_y\partial_x & \partial_y^2 & \partial_x^3 & \partial_y\partial_x^2 & \partial_y^2\partial_x & \partial_y^3 \cr
           x-y^3     & 0 & 1 & 0 & 0 & 0 & 0 & 0 & 0 & 0 &-1 \cr
           x^2       & 0 & 0 & 0 & 1 & 0 & 0 & 0 & 0 & 0 & 0 \cr
           x(x-y^3)  & 0 & 0 & 0 & 1 & 0 & 0 & 0 & 0 & 0 & 0 \cr
           y(x-y^3)  & 0 & 0 & 0 & 0 & 1 & 0 & 0 & 0 & 0 & 0 \cr
           x(x^2)    & 0 & 0 & 0 & 0 & 0 & 0 & 1 & 0 & 0 & 0 \cr
           y(x^2)    & 0 & 0 & 0 & 0 & 0 & 0 & 0 & 1 & 0 & 0 \cr
           x^2(x-y^3)& 0 & 0 & 0 & 0 & 0 & 0 & 1 & 0 & 0 & 0 \cr
           xy(x-y^3) & 0 & 0 & 0 & 0 & 0 & 0 & 0 & 1 & 0 & 0 \cr
           y^2(x-y^3)& 0 & 0 & 0 & 0 & 0 & 0 & 0 & 0 & 1 & 0 }. \]
   The kernel of this matrix has dimension 4, and a basis for it corresponds to the functionals $1$, $\partial_y$, $\partial_y^2$ and $\partial_y^3 + \partial_x$.  These form a basis for the truncated dual space $\dual^{(3)}[I]$.
\end{example}

\subsection{Mourrain Algorithm}
The second algorithm is due to Bernard Mourrain \cite{Mourrain:inverse-systems}.  We define the ``derivative'' of a dual functional with respect to a given variable $\partial_i$.  Let $\D_i:\dual \to \dual$ be the linear map defined by
	\[\D_i (\partial^\alpha) = \left\{ \begin{array}{ll}
		\partial^\alpha/\partial_i & \text{if}\; \partial_i|\partial^\alpha \\
	  0 & \text{otherwise}
	  \end{array}.\right.\]
Note that $\D_i \D_j = \D_j \D_i$.  Also for any $f \in \ring$ and $p \in \dual$ we have $p(x_if) = \D_i p(f)$.  It follows that if $p \in \dual[I]$ then $\D_i p \in \dual[I]$ for all $i$.  In fact there is a stronger result:
\begin{theorem}[\cite{Mourrain:inverse-systems}, Theorem 4.2]
	For any $p \in \dual$, $p \in \dual[I]$ if and only if $\D_i p \in \dual[I]$ for all $i$ and $p(f) = 0$ for all $f \in F$.
\end{theorem}

The dual elements with lead term of degree $d$ have derivatives which have lead term of degree $d-1$ or less.  This produces a way to build up $\dual[I]$ degree by degree.  Suppose $\beta_1,\ldots,\beta_r$ are a basis for $\dual^{(d-1)}[I]$.  Then for $p \in \dual^{(d)}[I]$, each derivative $\D_i p$ can be expressed in terms of this basis so
	\[ \D_i p = \sum_{j=1}^r \lambda^i_j \beta_j \]
for some coefficients $\lambda^i_j \in k$.  It can be shown that
	\[ p = \sum_{j=1}^r \lambda^1_j x_1\beta_j\big|_{x_2 = \cdots = x_n = 0} + \lambda^2_j x_2\beta_j\big|_{x_3 = \cdots = x_n = 0} + \cdots + \lambda^n_j x_n\beta_j \]
so the elements of $\dual^{(d)}[I]$ are all linear combinations of the terms of the form $x_i\beta_j\big|_{x_{i+1} = \cdots = x_n = 0}$.  However not all linear combinations work.  The fact that $\D_i\D_l p = \D_l\D_i p$ produces the relation
	\[ \sum_{j=1}^r (\lambda^l_j \D_i \beta_j - \lambda^i_j \D_l \beta_j) = 0 \quad \text{for each }1\leq i < l \leq n. \]
Since each $\D_i \beta_j$ is also in $\dual^{(d-1)}[I]$, it can be uniquely expressed in the basis $\beta_1,\ldots,\beta_r$ as $\D_i \beta_j = \sum_{l=1}^r \mu^i_{j,l} \beta_l$.  Then the above equation can be broken down into the linear relations
	\[ \sum_{j=1}^r (\lambda^l_j \mu^i_{j,m} - \lambda^i_j \mu^l_{j,m}) = 0 \]
	for each $1\leq i < l \leq n$, and $1\leq m \leq r$.  Finally $p(f_i) = 0$ for each generator $f_i$ produces another set of relations.  We can build a matrix with a row for each of these constraints and columns corresponding to the coefficients $\lambda^i_j$.  The kernel of this matrix corresponds to the space $\dual^{(d)}[I]$.

\section{Homogeneous Ideals}\label{homog}
If $I$ is homogeneous, then there is a criterion for deciding when all g-corners of $I$ have been found when searching degree by degree.  If $f$ and $g$ be homogeneous polynomials in $\ring$, then their S-pair is also homogeneous and
	\[ \deg S(f,g) = \deg \lcm(\LT f,\LT g) \leq \deg f + \deg g. \]
In addition, if $F$ is any set of homogeneous polynomials then $\NF_F(S(f,g))$ is also homogeneous with the same degree as $S(f,g)$.  Suppose that $F \subset I$ is a finite set of homogeneous polynomials with lead terms representing each of the g-corners up to degree $d$, but $F$ is not a standard basis.  By the Buchberger criterion, there is some pair $f_1,f_2 \in F$ with $g = \NF_F(S(f_1,f_2)) \neq 0$.  Then $\LT g$ is not divisible by any of the g-corners in $\LT F$, and $\deg g \leq \deg f_1 + \deg f_2 \leq 2d$, so there must be another g-corner with degree $\leq 2d$.
\begin{proposition}
\label{stopCrit}
	If $I$ is a homogeneous ideal and $C$ is the set of all g-corners of $I$ up to degree $d$ then either $C$ is the set of all g-corners of $I$ or there is an additional g-corner $m$ with $\deg m \leq \max_{a,b \in C} \deg \lcm(a,b) \leq 2d$.
\end{proposition}
So if finding bases for $\dual^{(c)}[I]$ at each $c$ up to $c = 2d$ reveals no g-corners with degree above $d$, then all g-corners have been found.  We would like to extend this idea to the more general case of non-homogeneous ideals.  Note that the bound of $2d$ can often be improved by taking $\max_{a,b \in C} \deg \lcm(a,b)$ instead.

Let $\ringh$ be the localization of $k[t,x_1,\ldots,x_n]$ by the maximal ideal $\langle t,x_1,\ldots,x_n\rangle$.  For $f \in \ring$ let $f^h \in \ringh$ denote the homogenization of $f$.  The \textit{\'ecart} of $f$ is the difference in total degree of the highest and lowest degree terms of $f$, or in other words the $t$-degree of the lead term of $f^h$.  For $p \in \dual$ let $p^h \in k[\partial_t,\partial_1,\ldots,\partial_n] = \dualh$ denote the homogenization of $p$.  Let $\psi:\dualh \to \dual$ be the function that dehomogenizes with respect to $t$.  So $\deh{q} = q|_{\partial_t = 1}$, and $\deh{p^h} = p$.  Abusing notation, we will also use $\psi$ to denote the dehomogenization function for $\ringh$.  For $I = \langle f_1,\ldots,f_s \rangle$, let $\Ih = \langle f_1^h,\ldots,f_s^h \rangle \subset \ringh$.  Note that $\Ih$ is not the same as the homogenization of $I$ and it will depend on the choice of generators of $I$.  We fix a particular set of generators $F = \{f_1,\ldots,f_s\}$ from here forward.  
It is easy to see that $\deh{\Ih} = I$, regardless of the choice of generators.

We take the local order on $\ringh$ to be some extension of the order $>$ on $\ring$ which is anti-graded and has $t > x_i$ for every $i$.  This ensures that for homogeneous $g \in \ringh$, $\LT \deh{g} = \dehw{\LT g}$.  The monomial order on $\dualh$ is taken to be the reverse of the order on $\ringh$, so $\partial_t \prec \partial_i$ for each $i$.  Note that for homogeneous $q \in \dualh$ this implies that $\LTg \deh{q} = \dehw{\LTg q}$.

\begin{theorem}
\label{homogCorner}
	If $G = \{g_1,\ldots,g_r\}$ is a homogeneous standard basis of $\Ih$, then $\deh{G} = \{\deh{g}_1,\ldots,\deh{g}_r\}$ is a standard basis of $I$.
\end{theorem}
\begin{proof}
 	Any homogeneous $g \in \Ih$ can be expressed as $g = \sum_i a_if_i^h$, so $\deh{g} = \sum_i \deh{a_i}f_i$ is in $I$.  Therefore $\deh{G} \subset I$.  Moreover $\dehw{\LT g} = \LT \deh{g} \in \LT I$.  For any $f = \sum_i b_if_i \in I$, let $d$ be the maximum degree of all $(b_if_i)^h$, and $c_i$ be the integer such that $t^{c_i}(b_if_i)^h$ has degree $d$.  Then $g = \sum_i t^{c_i}(b_if_i)^h \in \Ih$ is a homogeneous degree $d$ polynomial and has $\deh{g} = f$ and so $\dehw{\LT g} = \LT f$.  Therefore $\dehw{\LT \Ih} = \LT I$.  For any $m$ which is a g-corner of $I$, $t^am \in \LT \Ih$ for some $a$.  Taking $a$ to be the minimum such value, $t^am$ is a g-corner of $\Ih$.  Therefore some $g \in G$ has $t^am$ as its lead monomial, and $m$ is the lead monomial of $\deh{g}$, so $\ideal{\LT \deh{G}} = \LT I$.
\end{proof}

Therefore we can find the g-corners of $\Ih$ by calculating $\dualh^{(d)}[\Ih]$ for successive $d$, and using the stopping criterion for homogeneous ideals, and from this we can recover the g-corners of $I$, which determines the Hilbert function $H_I$.

\begin{example}
\label{mainEx1}
	Let $I$ be the ideal
		\[ I = \ideal{x^2-xy^3, x^4} \subset \B C[x,y]_{\langle x,y \rangle}. \]
	All terms of the generators have degree 4 or less and the Hilbert function $H_I(d) = 1$ for $d < 10$.  Finding the truncated dual of $I$ at several degrees, one might be tempted to conclude that the Hilbert function stabilizes at 1.  However, at $d=10$ there is a new g-corner, and $H_I(d) = 1$ for all $d \geq 10$.  A reduced standard basis of $I$ is $\{x^2-xy^3, xy^9\}$.

	We look instead at the ideal
		\[ \Ih = \ideal{t^2x^2-xy^3, x^4} \subset \B C[t,x,y]_{\langle t,x,y \rangle} \]
	which has reduced standard basis $\{t^2x^2-xy^3, x^4, x^3y^3, x^2y^6, xy^9\}$.  The g-corners of $\Ih$ occur at closer intervals in degree.  Beyond the highest degree of the generators of $\Ih$, no g-corner has degree more than twice that of the smaller degree g-corners.
\end{example}

\section{Eschewing Homogenization}\label{dehomog}

Although the method described in the previous section works, we would like to discover the g-corners of $I$ without explicitly homogenizing the ideal, since this may introduce unnecessary numerical error to the process.  Additionally, introducing an extra variable causes a significant increase in the computation time of the dual space algorithms, which we would like to avoid.  To get around homogenization we can take advantage of the particular structure of $\Ih$.

\begin{definition}
	For $f \in \ring$, the \'ecart of $f$ is the difference in degree between the highest degree term and the lowest degree term of $f$.
\end{definition}

\begin{theorem}
\label{maxEcart}
	Let $e$ be the maximum \'ecart of the generators $f_1,\ldots,f_s$ of $I$, and let $q \in \dualh$ be homogeneous with lead term having $\partial_t$-degree at least $e$.  Then $q \in \dualh[\Ih]$ if and only if $\deh{q} \in \dual[I]$.
\end{theorem}
\begin{proof}
	Suppose $\deh{q} \in \dual[I]$.  For any $f \in I$, $q(t^b f^h) = \deh{q}(f) = 0$ if $q$ and $t^b f^h$ have the same total degree, otherwise $q(t^b f^h) = 0$ since they have no compatible monomials.  Note that any element of $\Ih$ can be expressed as the sum of homogeneous polynomials of the form $t^b f^h$ with $f \in I$, and $q$ annihilates all such polynomials, so $q \in \dualh[\Ih]$.

	To prove the other direction, we use induction on the total degree $d$ of $q$.  For any $d < e$ we have that $\deh{q} \in \dual[I]$ implies $q \in \dualh[\Ih]$ vacuously since there are no functionals $q$ with total degree $d$ and lead term having $\partial_t$-degree at least $e$.  Suppose for some $d$ that for all homogeneous $q \in \dualh[\Ih]$ with degree at most $d$ and lead term with $\partial_t$-degree at least $e$ that $\deh{q} \in \dual[I]$.  Fix some $q \in \dualh[\Ih]$ with degree $d+1$ and lead term with $\partial_t$-degree at least $e$.  To show that $\deh{q} \in \dual[I]$ it is sufficient to show i) that the first derivative $d_i\deh{q}$ is in $\dual[I]$ for each dual variable $\partial_i$ and ii) that $\deh{q}(f_j) = 0$ for all generators $f_j$.

	i) It is easy to check that differentiation with respect to any $\partial_i$ commutes with dehomogenization: $d_i\deh{q} = \dehw{d_i q}$.  The functional $d_i q$ has total degree $d$ and every term of $q$ has $\partial_t$-degree at least as large as the $\partial_t$-degree of the lead term, which is $\geq e$, so the lead term of $d_i q$ has $\partial_t$-degree at least $e$, or $d_i q = 0$.  Therefore $d_i q \in \dualh[\Ih]$ so $d_i(\deh{q}) \in \dual[I]$.
	
	ii) $\deh{q}(f_j) = \partial_t^a q(t^b f_j^h)$ for any values of $a$ and $b$ for which the total degrees of $\partial_t^a q$ and $t^b f_j^h$ are equal.  If $a > b$ then $\partial_t^a q(t^b f_j^h) = 0$ since every term of $\partial_t^a q$ has $\partial_t$-degree at least $a + e$ and every term of $t^b f_j^h$ has $t$-degree at most $b + e$.  If $a\leq b$ then $\deh{q}(f_j) = q(t^{b-a} f_j^h) = 0$ since $q \in \dualh[\Ih]$ and $t^{b-a} f_j^h \in \Ih$.
\end{proof}

\begin{corollary}
	Let $e$ be the maximum \'ecart of the generators of $I$.  Every g-corner of $\Ih$ has $t$-degree $\leq e$.
\end{corollary}
\begin{proof}
	If $t^ax^\alpha$ is a g-corner of $\Ih$ then $x^\alpha \in \LT I$ by Theorem \ref{homogCorner}, and so $\partial^\alpha \notin \LTg \dual[I]$.  Therefore $\partial_t^b\partial^\alpha \notin \LTg \dualh[\Ih]$ for all $b \geq e$, so $t^bx^\alpha \in \LT \Ih$ for all $b \geq e$.  Since $a$ is the minimum value for which $t^ax^\alpha \in \LT \Ih$, it must be that $a \leq e$.
\end{proof}

At and above $t$-degree $e$, the dual space of $\Ih$ looks just like the dual space of $I$ (after dehomogenizing).  At and below $t$-degree $e$ is where g-corners of $\Ih$ may occur and this information will be used to decide what degree to calculate the dual space up to.  Let $\dualh^d[\Ih]$ denote the subspace of $\dualh[\Ih]$ with degree exactly $d$.

\begin{corollary}
	$\dual^{(d)}[I] \subset \dehw{\dualh^d[\Ih]}$ and the subspace of $\dehw{\dualh^d[\Ih]}$ of elements with lead term of degree $d-e$ or less is equal to $\dual^{(d-e)}[I]$.
\end{corollary}
	
\begin{proof}
	For any $p \in \dual^{(d)}[I]$ let $q$ be the homogenization of $p$ to degree $d$, that is $q = \partial_t^a p^h$ with $a$ chosen so that $q$ has degree $d$.  For any $g \in \Ih$, $q(g) = 0$ trivially if $g$ does not have degree $d$, and otherwise $q(g) = p(\deh{g}) = 0$ since $\deh{g} \in I$.  Therefore $q \in \dualh^d[\Ih]$, which proves the first part of the statement.  An element in $\dehw{\dualh^d[\Ih]}$ with lead term of degree $\leq d-e$ is the dehomogenization of an element $p \in \dualh^d[\Ih]$ with lead term with $\partial_t$-degree $\geq e$, so $\deh{p} \in \dual[I]$. 
\end{proof}

\begin{figure}[ht]
  \centering
  \includegraphics[width=.9\columnwidth]{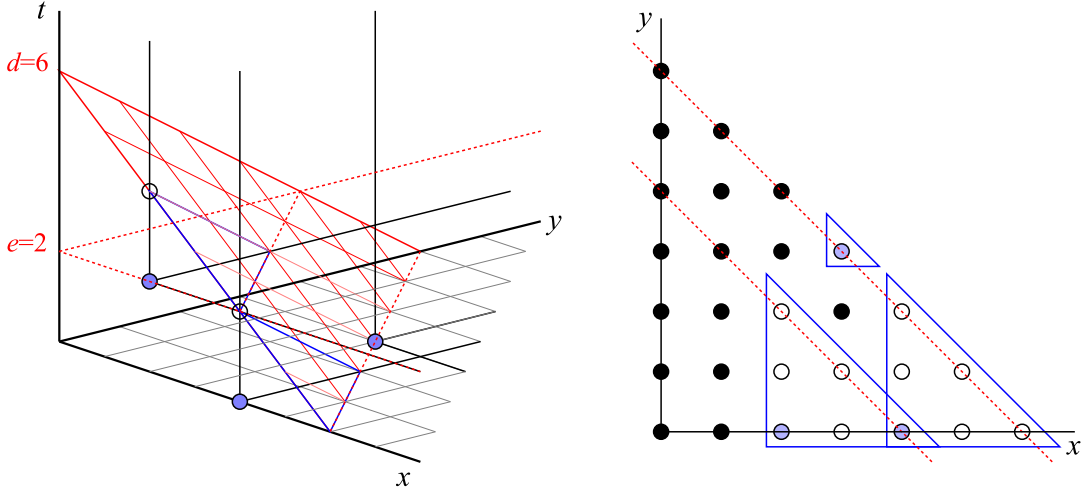}
  \caption{Here $I = \ideal{x-y^3,x^3}$.  On the left is the monomial lattice for $\Ih$, sliced at $d=6$.  All g-corners occur at or below the plane of $t$-degree $2$.  On the right is the slice after dehomogenization where black dots represent lead monomials of $\dehw{\dualh^6[\Ih]}$}
\end{figure}

\section{The Sylvester Dual}\label{sylvester}

\begin{definition}
	The Sylvester array $\BF S(F,d)$ for a set of generators $F$ of ideal $I$, and a degree $d$ is the coefficient matrix of all monomial multiples of the generators in $F$, $x^\alpha f$ such that every term of $x^\alpha f$ has degree $d$ or less.  The columns correspond to each of the monomials up to degree $d$.
\end{definition}

The Sylvester array $\BF S(F,d)$ is similar to the Macaulay array $\BF M(F,d)$ but instead of having a row for every monomial multiple of a generator that has any terms of degree $\leq d$, it only includes the ones that have all terms of degree $\leq d$.  The kernel of $\BF M(F,d)$ corresponded to $\dual^{(d)}[I]$.  The kernel of $\BF S(F,d)$ also defines a subspace of $\dual^{(d)}$, which will be denoted $\sdual^{(d)}[F]$.  Note that unlike the truncated dual space, $\sdual^{(d)}[F]$ depends on the set of generators for $I$.  Also unlike the truncated dual space, it is not generally true that $\sdual^{(d)}[F] \subset \sdual^{(d+1)}[F]$.

\begin{example}
	As in the Macaulay array example (Example \ref{MacExample}) let $F = \{x-y^2, x^2\} \subset k[x,y]$.  Then the Sylvester array $\BF S(F,3)$ is
		\[ \bordermatrix{ & 1 & \partial_x & \partial_y & \partial_x^2 & \partial_y\partial_x & \partial_y^2 & \partial_x^3 & \partial_y\partial_x^2 & \partial_y^2\partial_x & \partial_y^3 \cr
           x-y^3     & 0 & 1 & 0 & 0 & 0 & 0 & 0 & 0 & 0 &-1 \cr
           x^2       & 0 & 0 & 0 & 1 & 0 & 0 & 0 & 0 & 0 & 0 \cr
           x(x^2)    & 0 & 0 & 0 & 0 & 0 & 0 & 1 & 0 & 0 & 0 \cr
           y(x^2)    & 0 & 0 & 0 & 0 & 0 & 0 & 0 & 1 & 0 & 0 }. \]
   The kernel of this matrix, $\sdual^{(3)}[F]$, has dimension 6, with basis $\{1, \partial_y, \partial_y\partial_x, \partial_y^2, \partial_y^2\partial_x, \partial_y^3 + \partial_x\}$.  Note that the rows of $\BF S(F,d)$ are a subset of the rows of $\BF M(F,d)$, so $\dual^{(d)}[I]$ is contained in $\sdual^{(d)}[F]$.
\end{example}

\begin{theorem}
	$\sdual^{(d)}[F] = \dehw{\dualh^d[\Ih]}$.
\end{theorem}
\begin{proof}
	$\dualh^d[\Ih]$ is exactly the set of degree-$d$ homogeneous functionals in $\dualh$ that annihilate the degree-$d$ homogeneous polynomials of $\Ih$.  The space of all degree-$d$ homogeneous elements of $\Ih$ is spanned by the elements of the form $t^a x^\alpha f^h$ where $f \in F$ is part of the generating set of $I$ and $a + |\alpha| + \deg f^h = d$.  Note that $\deg f^h$ is equal to the maximum degree of any term in $f$.  Dehomogenizing everything, $\dehw{\dualh^d[\Ih]}$ is the set of functionals in $\dual$ that annihilate all polynomials of the form $x^\alpha f$ such that all terms have degree $\leq d$.  These polynomials are exactly the ones which form the rows of $\BF S(F,d)$.
\end{proof}

This relationship provides an alternate way to prove Theorem \ref{maxEcart}.
\begin{proof}[(Alternate proof of Theorem \ref{maxEcart})]
	Let $q \in \dualh$ be homogeneous with total degree $d$ and lead term with $\partial_t$-degree $\geq e$.  This implies that $\psi(q)$ has all terms of degree $d-e$ or less.  Suppose $q \in \dualh[\Ih]$, so then $\deh{q} \in \sdual^{(d)}[F]$ by the previous theorem.  This means $\deh{q}$ annihilates each $x^\alpha f$ with all terms of degree $d$ or less and $f \in F$.  For $x^\beta f$ with some term of degree $> d$, the degree of the lead term of $x^\beta f$ must be $> d-e$ because the \'ecart of $f$ is at most $e$.  Therefore $\deh{q}$ also annihilates $x^\beta f$, since they have no terms in common.  $\deh{q}$ annihilates all terms of the form $x^\alpha f$ for any $x^\alpha$ and any $f \in F$, so $\deh{q} \in \dual[I]$.
	
	Suppose $\deh{q} \in \dual[I]$.  Then $\deh{q}$ annihilates all $x^\alpha f$ for $f \in F$, which implies $\deh{q} \in \sdual^{(d)}[I]$, so $q \in \dualh[\Ih]$.
\end{proof}

$\sdual^{(d)}[F]$ can be calculated at each degree $d$ without homogenizing, and captures all the information of the homogenized dual space.  In the lattice of monomials, $\LTg \dualh^d[\Ih]$ can be considered as a slice of $\LTg \dualh[\Ih]$ at degree $d$.  For each g-corner $m$ of $\Ih$, the monomial multiples of $m$ will be missing from $\LTg \dualh^d[\Ih]$.  This slice of the cone generated by $m$ will appear as a truncated cone of missing monomials in $\sdual^{(d)}[F]$, starting at the monomial $\deh{m}$ and extending out to all multiples of $\deh{m}$ up to degree $d - a$ where $a$ is the $t$-degree of $m$.  The monomials missing from $\sdual^{(d)}[F]$ are the union of all the truncated cones generated by the all the g-corners of $\Ih$ up to degree $d$.  

In $\dual^{(d)}[I]$ there is also a cone of missing monomials at $\deh{m}$ for each g-corner $m$ of $\Ih$, but in this case the cone extends all the way to degree $d$.

Suppose $\{g_1,\ldots,g_r\}$ is the standard basis of $\Ih$.  The subspace of $\Ih$ with degree $d$ is spanned by polynomials of the form $m g_i$ where $m$ is any monomial with degree $d - \deg g_i$.  For any particular $g_i$, the possible values of $\deh{m}$ are all the monomials in $\ring$ up to degree $d - \deg g_i$.  In the lattice of monomials of $\ring$, the possible values of $\LT(\dehw{mg_i})$ form a truncated cone, starting at $\LT(\deh{g_i})$ and extending out to all multiples up to degree $d - (\deg \LT(g_i) - \deg \LT(\deh{g_i}))$.  The value of $\deg \LT(g_i) - \deg \LT(\deh{g_i})$ is the $t$-degree of $\LT(g_i)$ and it is at most $e$.  The monomials excluded from $\init(\dehw{\dualh^d[\Ih]})$ are the union of the truncated cone from each $\LT(\deh{g_i})$.  In contrast the monomials excluded from $\init(\dual^{(d)}[I])$ are \emph{all} multiples of $\LT(\deh{g_i})$, so this picture is similar but the excluded cones extend all the way to degree $d$.

\begin{figure}[ht]\label{sylFig}
  \centering
  \includegraphics[width=.85\columnwidth]{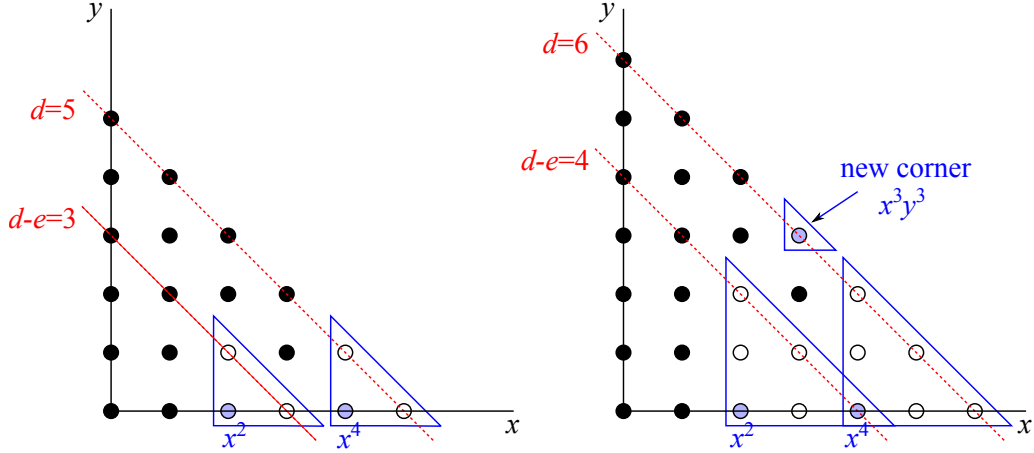}
  \caption{Black dots represent lead monomials in $\sdual^{(d)}[F]$ for $d = 5$ and $d = 6$ with $F = \{x^2-xy^3,x^4\}$.  The missing monomials form truncated cones (blue triangles).  At $d=6$ a new g-corner is discovered at $x^3y^3$.}
\end{figure}

If all g-corners of $\Ih$ are known up to degree $d-1$, then it can be calculated exactly which monomials should be missing in $\LTg \sdual^{(d)}[F]$ if there are no additional g-corners at degree $d$.  By calculating a basis for $\sdual^{(d)}[F]$, whichever additional monomials are missing must be new g-corners of $\Ih$ with degree $d$.  Therefore by calculating $\LTg \sdual^{(d)}[F]$ for each successive $d$ it is possible to discover the g-corners of $\Ih$ at each degree.  We can determine when all the g-corners of $\Ih$ have been found using Proposition \ref{stopCrit}.  The g-corners of $\Ih$ determine the g-corners of $I$, which fully determines the Hilbert function $H_I$.  The process of finding a basis for $\sdual^{(d)}[F]$ also produces the truncated dual space of $I$, since at each $d$ the set the elements of $\sdual^{(d)}[F]$ with lead term at most $d-e$ is $\dual^{(d-e)}[I]$.

\begin{algorithm}\label{sylAlg}
\quad \\
Inputs: generators $F = \{f_1,\ldots,f_s\}$ of ideal $I$.\\
Outputs: monomials $C = \{c_1,\ldots,c_r\}$ corresponding to g-corners of $\Ih$.
\begin{algorithmic}
\STATE $B := \{\};$ \quad List of pairs of a g-corner and a degree.
\STATE $d := 0;$
\STATE $d_{max} := 2\max_i\{\deg \LT f_i\};$
\WHILE{$d \leq d_{max}$}
	\STATE build Sylvester array $\BF S(F,d)$;
	\STATE $Sdual :=$ basis for $\ker \BF S(F,d)$;
	\STATE reduce $Sdual$ so each element has unique lead term;
	\STATE $B_d := \{\};$
	\FOR{monomials $m\notin \LT Sdual$ with $\deg m \leq d$}
		\IF{for all $(c_i,d_i) \in B$, either $c_i\nmid m$ or $\deg(m/c_i) > d-d_i$}
			\STATE append $(m,d)$ to $B_d$;
		\ENDIF
	\ENDFOR
	\IF{$B_d \neq \{\}$ \AND $d_{max} < 2d$}
		\STATE $d_{max} \gets 2d;$
	\ENDIF
	\STATE append $B_d$ to $B$;
	\STATE $d \gets d+1;$
\ENDWHILE
\STATE $C:=$ list of monomials from pairs in $B$;
\RETURN $C;$
\end{algorithmic}
\end{algorithm}

Given the output $C$ of this algorithm, to find the g-corners of $I$, simply remove all monomials $c_i \in C$ that are divisible by some other monomial in $C$.

\begin{remark}
	When dealing with numerically obtained inputs, the algorithm should use singular value decomposition to calculate a basis for the kernel of the Sylvester array.  This method is numerically stable, while using Gaussian elimination is not.
\end{remark}

$\sdual^{(d)}[F]$ reveals not only the g-corners of $\Ih$ at degree $d$ (and therefore the g-corners of $I$), but it can also be used to find the corresponding standard basis elements of $I$.  Additionally $\sdual^{(d)}[F]$ can answer the ideal membership test for polynomials with all terms of degree $d$ or less.  These two facts are encompassed by the following proposition.
\begin{proposition}
	If $f \in \ring$ is a polynomial with no terms exceeding degree $d$ and $p(f) = 0$ for all $p \in \sdual^{(d)}[F]$, then $f \in I$.  For each monomial $\partial^\alpha$ of degree $\leq d$ that is not in $\LTg \sdual^{(d)}[F]$, there is some polynomial $f \in \ring$ satisfying the above conditions with $\LT f = x^\alpha$.
\end{proposition}
\begin{proof}
	Suppose $f \in \ring$ is a polynomial with no terms exceeding degree $d$ and $p(f) = 0$ for all $p \in \sdual^{(d)}[F]$.  Let $g$ be the homogenization of $f$ to degree $d$ (i.e. $g = t^af^h$ where $a = d - \deg f^h$).  Then $q(g) = 0$ for all $q \in \dualh^d[\Ih]$ so $g \in \Ih$.  Therefore $f = \deh{g} \in I$.
	
	If $\partial^\alpha$ with degree $\leq d$ is not in $\LTg \sdual^{(d)}[F]$, then its homogenization $\partial_t^b\partial^\alpha$ to degree $d$ is not in $\LTg \dualh[\Ih]$ so there is some homogeneous $g \in \Ih$ with $\LT g = t^bx^\alpha$.  Therefore $\deh{g}$ has lead term $x^\alpha$ and is annihilated by $\sdual^{(d)}[F]$.
\end{proof}

Supposing a basis for $\sdual^{(d)}[F]$ has been calculated, build the coefficient matrix of these basis elements with columns for each of the monomials up to degree $d$.  The kernel of this matrix corresponds to the polynomials in $\ring$ with all terms of degree $\leq d$ that are annihilated by $\sdual^{(d)}[F]$.  Let $m_1,\ldots,m_r$ be the set of g-corners of $\Ih$.  If monomial $m_i \in \dualh$ has degree $d$, there must be some polynomial $h_i$ in this kernel with $\LT g_i = \deh{h_i}$.  Collecting the polynomials found this way for each g-corner of $\Ih$ produces a set $H = \{h_1,\ldots,h_r\} \subset I$, with $\LT H = \LT I$ so $H$ is a standard basis of $I$.

\begin{algorithm}
\quad \\
Inputs: Basis for the Sylvester dual $Sdual$ at some degree $d$, and a g-corner $c$ of $I$ found at degree $d$.\\
Outputs: Polynomial $p \in I$ with $\LT p = c$.
\begin{algorithmic}
\STATE $monomials :=$ list of monomials $m \in \ring$ with $\deg m \leq d$ and $m < c$;
\STATE $M :=$ coefficient matrix of elements in $Sdual$, with columns only for the monomials in $monomials$;
\STATE $K := \ker M$;
\RETURN an element $p$ of $K$ with $\LT p = c$;  \quad Such an element is guaranteed to exist.
\end{algorithmic}
\end{algorithm}

\begin{example}
We continue Example \ref{mainEx1} with $F = \{x^2-xy^3, x^4\}$ and $I = \ideal{F}$, and run through the algorithm for finding the g-corners of $I$ and a standard basis, this time using Sylvester arrays instead of homogenizing the generators.  For $d = 0,1,2,3$ the Sylvester arrays are empty because there are no multiples of $x^2-xy^3$ or $x^4$ which have all terms of degree 3 or less.  Therefore bases for the first four Sylvester dual spaces are
	\[ \begin{array}{rl}
		\sdual^{(0)}[F]: & \{1\},\\
		\sdual^{(1)}[F]: & \{1, \partial_y, \partial_x\},\\
		\sdual^{(2)}[F]: & \{1, \partial_y, \partial_x, \partial_y^2, \partial_y\partial_x, \partial_x^2\},\\
		\sdual^{(3)}[F]: & \{1, \partial_y, \partial_x, \partial_y^2, \partial_y\partial_x, \partial_x^2, \partial_y^3, \partial_y^2\partial_x, \partial_y\partial_x^2, \partial_x^3\}.
	\end{array} \]

At degree 4, the Sylvester array $\BF S(F,4)$ has rows for $x^2-xy^3$ and $x^4$.  The kernel of this matrix $\sdual^{(4)}[F]$ has basis
	\[ \begin{array}{rr}
		\sdual^{(4)}[F]: & \{1, \partial_x, \partial_y, \partial_y\partial_x, \partial_y^2, \partial_x^3, \partial_y\partial_x^2, \partial_y^2\partial_x, \partial_y^3,\\
		& \partial_y\partial_x^3, \partial_y^2\partial_x^2, \partial_y^3\partial_x+\partial_x^2, \partial_y^4\}.
	\end{array} \]
The monomials $\partial_x^2$ and $\partial_x^4$ are both missing from the set of lead monomials of these basis elements.  Since there were no previous g-corners found, each of $x^2$ and $x^4$ must be the dehomogenization of a g-corner of $\Ih$.  They are recorded as potential g-corners of $I$ along with the degree they were found at, which is 4.  To find standard basis elements corresponding to these g-corners, we construct the coefficient matrix of the basis for $\sdual^{(4)}[F]$, and try to find elements of the kernel with lead terms $x^2$ and $x^4$.  The polynomials $x^4$ and $x^2 - xy^3$ may be produced this way.

At degree 5, $\BF S(F,5)$ has rows for $x^2-xy^3$, $x^4$, $x^3-x^2y^3$, $x^5$, $x^2y-xy^4$, $x^4y$.  A basis for the Sylvester dual is 
	\[ \begin{array}{rr}
		\sdual^{(5)}[F]: & \{1, \partial_x, \partial_y, \partial_y\partial_x, \partial_y^2, \partial_y^2\partial_x, \partial_y^3, \partial_y\partial_x^3, \partial_y^2\partial_x^2,\\
		& \partial_y^3\partial_x+\partial_x^2, \partial_y^4, \partial_y^2\partial_x^3, \partial_y^3\partial_x^2, \partial_y^4\partial_x+\partial_y\partial_x^2, \partial_y^5\}.
	\end{array} \]
Since $x^2$ and $x^4$ were g-corners found at degree 4, the multiples up to 1 degree higher will be missing at degree 5.  This accounts for all the monomials missing from this basis for $\sdual^{(5)}[F]$, which are $\partial_x^2$, $\partial_x^3$, $\partial_y\partial_x^2$, $\partial_x^4$, $\partial_x^5$ and $\partial_y\partial_x^4$, so there are no new g-corners here.

At degree 6, a basis for the Sylvester dual is
	\[ \begin{array}{rr}
		\sdual^{(6)}[F]: & \{1, \partial_x, \partial_y, \partial_y\partial_x, \partial_y^2, \partial_y^2\partial_x, \partial_y^3, \partial_y^3\partial_x+\partial_x^2,\\
		& \partial_y^4, \partial_y^2\partial_x^3, \partial_y^3\partial_x^2+\partial_x^3, \partial_y^4\partial_x+\partial_y\partial_x^2,\\
		& \partial_y^5, \partial_y^4\partial_x^2+\partial_y\partial_x^3, \partial_y^5\partial_x+\partial_y^2\partial_x^2, \partial_y^6\}.
		\end{array} \]
The monomials missing from the lead terms are $\partial_x^2$, $\partial_x^3$, $\partial_y\partial_x^2$, $\partial_x^4$, $\partial_y\partial_x^3$, $\partial_y^2\partial_x^2$, $\partial_x^5$, $\partial_y\partial_x^4$, $\partial_x^6$, $\partial_y\partial_x^5$, $\partial_y^2\partial_x^4$ and $\partial_y^3\partial_x^3$.  All but $\partial_y^3\partial_x^3$ corresponds to a g-corner $x^2$ or $x^4$ or a multiple of these by a monomial up to degree 2.  Therefore $x^3y^3$ is a new g-corner, which we store along with the degree 6 at which it was found.  (See Figure 4 for a diagram of this step.)  A corresponding standard basis element is $x^3y^3$.

Continuing this process at each degree, the g-corner $x^2y^6$ is found at $d = 8$ and $xy^9$ is found at $d = 10$.  Corresponding standard basis elements are $x^2y^6$ and $xy^9$.  No additional g-corners are found searching up to degree 20, so all the g-corners of $I$ are among $x^2,x^4,x^3y^3,x^2y^6,xy^9$.  The monomials that are multiples of other monomials in the list can be dropped, leaving $x^2$ and $xy^9$.  The corresponding standard basis elements found for these g-corners were $x^2-xy^3$ and $xy^9$ so this is the reduced standard basis for $I$.  Finally the Hilbert function can easily be recovered from the set of g-corners, which is $H_I(d) = 2$ for $1 \leq d \leq 10$, and $H_I(d) = 1$ for $d = 0$ and $d > 10$.
\end{example}

We have implemented Algorithm \ref{sylAlg} for finding the g-corners of an ideal using the Sylvester dual, and this algorithm for recovering a standard basis of the ideal, in the computer algebra system {\em Macaulay2}.  This implementation is contained in the package "NumericalHilbert," which can be found at 
\\ \url{http://people.math.gatech.edu/~rkrone3/NHcode.html}.

$\sdual^{(d)}[F]$ can also be calculated using a variation of the algorithm presented by Mourrain.  The Mourrain algorithm works by finding the elements of $\dualh$ whose derivatives are in the span of the previously found dual basis elements, and which annihilate the generators $f_1^h,\ldots,f_s^h$.  Here we can take advantage of the fact that there exists a homogeneous basis of the truncated dual space.  Each of the derivatives of a degree $d$ dual element has degree $d-1$, so to find the degree $d$ elements we need only to consider the elements found in the preceding step.  The Macaulay2 package also contains an implementation of this version of the algorithm.  Using the Mourrain algorithm is more efficient than using the Sylvester array strategy in most cases because the matrices involved grow with the dimension of the dual space rather than with the dimension of the entire ring up to the given degree.

The algorithm produces:
\begin{itemize}
 \item a basis for the dual space truncated to the degree that the algorithm needed to calculate up to
 \item the g-corners of the ideal (or optionally a standard basis of the ideal)
 \item a bound on the regularity of the Hilbert function
 \item the values of the Hilbert function up to the regularity bound
 \item the Hilbert polynomial which defines the values above the regularity
\end{itemize}
Note that the exact regularity of the Hilbert function can easily be computed from this data by comparing the returned values of the Hilbert function to the Hilbert polynomial at each degree below the bound.

\begin{example}
In this example we demonstrate the functionality of the Macaulay2 package.  Let $I$ be the ideal defined by the system Cyclic4, given by the following polynomials in $\B C[x_1,x_2,x_3,x_4]$:
\[ F = \{x_1 + x_2 + x_3 + x_4, \; x_1x_2 + x_2x_3 + x_3x_4 + x_4x_1, \]
\[ x_1x_2x_3 + x_1x_2x_4 + x_1x_3x_4 + x_2x_3x_4, \; x_1x_2x_3x_4 - 1\}. \]
The variety consists of two irreducible curves, along with 8 embedded 0-dimensional components.  One of the embedded components is $(-1,1,1,-1)$.  However suppose we were not aware of this, and only had an approximate numerical value for this point.  Below we use an approximation with small error that was obtained using a numerical solver.

\begin{Macaulay2}
\begin{verbatim}
i1 : loadPackage "NumericalHilbert";

i2 : R = CC[x_1..x_4, MonomialOrder=>{Weights=>{-1,-1,-1,-1}},Global=>false];

i3 : F = matrix{{x_1 + x_2 + x_3 + x_4,
                 x_1*x_2 + x_2*x_3 + x_3*x_4 + x_4*x_1,
                 x_2*x_3*x_4 + x_1*x_3*x_4 + x_1*x_2*x_4 + x_1*x_2*x_3,
                 x_1*x_2*x_3*x_4 - 1}};

i4 : P = {-1.0-.53734e-17*ii,  1.0-.20045e-16*ii,
           1.0+.89149e-17*ii, -1.0+.18026e-17*ii};

i5 : dualInfo(F, Point=>P, Tolerance=>1e-4)

o5 = ({1, - 1x  + x , - 1x  + x , ... },
              1    3      2    4
                2
      {x , x , x , x x }, 5, {1, 2, 1, 1, 1}, 1)
        1   2   3   3 4
\end{verbatim}
\end{Macaulay2}
The output at o5 is a sequence consisting of the information listed above the start of the example.  First is a basis for $\dual^{(d)}[I]$ but we do not reproduce the full output here in the paper for reasons of space.  Shown is only $\dual^{(1)}[I]$.  Note that the dual elements are written in terms of the original variables of $\ring$ even though this is an abuse of notation.  Next in the sequence is the list of g-corners generating $\LT I$, $x_1, x_2, x_3^2, x_2x_4$.  The last two entries of the sequence are the first five values of the Hilbert function $(1,2,1,1,1)$ followed by the Hilbert polynomial, which is 1.  We see from the output that the Hilbert function is $H_I(d) = 1$ for all $d$ except $H_I(1) = 2$.  The fact that the Hilbert polynomial is a non-zero constant indicates that indeed the point in question sits on a 1-dimensional component of the variety.  All of these results agree with the values that are obtained by symbolic computation at the point $(-1,1,1,-1)$, but we did not need to know 
the exact value of this point.
\end{example}

\begin{example}
We give another example, this time where the point of interest is not rational so it must be approximated.  Define $I$ to be the ideal generated by $F = \{(x_1^2 + x_2^2 + x_3^2 - 1)(x_1-x_2), (x_1-x_2)^3\} \subset \B C[x_1,x_2,x_3]$.  The variety is the plane $x_1 - x_2 = 0$, but there is an embedded curve which passes through the point $(\sqrt{2},\sqrt{2},0)$.  We will use a numerical approximation of this point to find the Hilbert function of the ideal localized here.
\begin{Macaulay2}
\begin{verbatim}
i1 : loadPackage "NumericalHilbert";

i2 : R = CC[x_1..x_3, MonomialOrder=>{Weights=>{-1,-1,-1}},Global=>false];

i3 : F = matrix{{(x_1^2 + x_2^2 + x_3^2 - 1)*(x_1 - x_2),
                 (x_1 - x_2)^3}};

i4 : P = {0.7071068,  0.7071068, 0};

i5 : dualInfo(F, Point=>P, Tolerance=>1e-4)

o5 = ({1, x , x , 1x , ... },
           1   2    3
        2     2
      {x , x x }, 4, {1, 3, 5, 6}, i + 3)
        1   1 2
\end{verbatim}
\end{Macaulay2}
We again truncate here the list of basis elements of the dual space of the ideal for space reasons.  The algorithm finds that 4 is a bound on the regularity of the Hilbert function, and the values of $H_I(i)$ for $i = 0,1,2,3$ are $1,3,5,6$ respectively.  Beyond that point, $H_I(i) = i + 3$.  The fact that the Hilbert polynomial is linear indicates that the largest component passing through $(\sqrt{2},\sqrt{2},0)$ has dimension 2.
\end{example}

\bibliographystyle{plain}
\bibliography{NumDualAlgs}

\begin{thebibliography}{10}

\bibitem{DBLP:journals/jsc/Apel98}
Joachim Apel.
\newblock The theory of involutive divisions and an application to hilbert
  function computations.
\newblock {\em J. Symb. Comput.}, 25(6):683--704, 1998.

\bibitem{Bates-et-al:local-dimension-test}
Daniel~J. Bates, Jonathan~D. Hauenstein, Chris Peterson, and Andrew~J. Sommese.
\newblock A numerical local dimensions test for points on the solution set of a
  system of polynomial equations.
\newblock {\em SIAM J. Numer. Anal.}, 47(5):3608--3623, 2009.

\bibitem{Cox-Little-OShea:using}
David~A. Cox, John Little, and Donal O'Shea.
\newblock {\em Using algebraic geometry}, volume 185 of {\em Graduate Texts in
  Mathematics}.
\newblock Springer, New York, second edition, 2005.

\bibitem{DZ-05}
B.H. Dayton and Z.~Zeng.
\newblock Computing the multiplicity structure in solving polynomial systems.
\newblock In M.~Kauers, editor, {\em Proceedings of the 2005 International
  Symposium on Symbolic and Algebraic Computation}, pages 116--123. ACM, 2005.

\bibitem{M2www}
Daniel~R. Grayson and Michael~E. Stillman.
\newblock Macaulay 2, a software system for research in algebraic geometry.
\newblock Available at http://www.math.uiuc.edu/Macaulay2/.

\bibitem{Singular-book-02}
Gert-Martin Greuel and Gerhard Pfister.
\newblock {\em A {\bf {S}ingular} introduction to commutative algebra}.
\newblock Springer, Berlin, extended edition, 2008.
\newblock With contributions by Olaf Bachmann, Christoph Lossen and Hans
  Sch{\"o}nemann, With 1 CD-ROM (Windows, Macintosh and UNIX).

\bibitem{KreuzerRobbianoBook2}
Martin Kreuzer and Lorenzo Robbiano.
\newblock {\em Computational commutative algebra. 2}.
\newblock Springer-Verlag, Berlin, 2005.

\bibitem{DBLP:conf/issac/Leykin08}
Anton Leykin.
\newblock Numerical primary decomposition.
\newblock In {\em ISSAC}, pages 165--172, 2008.

\bibitem{LVZ-higher}
Anton Leykin, Jan Verschelde, and Ailing Zhao.
\newblock Higher-order deflation for polynomial systems with isolated singular
  solutions.
\newblock In {\em Algorithms in algebraic geometry}, volume 146 of {\em IMA
  Vol. Math. Appl.}, pages 79--97. Springer, New York, 2008.

\bibitem{Macaulay:modular-systems}
F.~S. Macaulay.
\newblock {\em The algebraic theory of modular systems}.
\newblock Cambridge Mathematical Library. Cambridge University Press,
  Cambridge, 1994.
\newblock Revised reprint of the 1916 original, With an introduction by Paul
  Roberts.

\bibitem{mora2005solving}
T.~Mora.
\newblock {\em Solving Polynomial Equation Systems II: Macaulay's Paradigm and
  Gr{\"o}bner Technology}.
\newblock Encyclopedia of Mathematics and Its Applications. Cambridge
  University Press, 2005.

\bibitem{Mourrain:inverse-systems}
B.~Mourrain.
\newblock Isolated points, duality and residues.
\newblock {\em J. Pure Appl. Algebra}, 117/118:469--493, 1997.
\newblock Algorithms for algebra (Eindhoven, 1996).

\end{thebibliography}

\end{document}